\documentclass[11pt, reqno]{amsart}

\usepackage{amssymb, latexsym, mathrsfs,   color }
\usepackage[all]{xy}
\usepackage[colorlinks=true, pdfstartview=FitV,linkcolor=blue,citecolor=blue,urlcolor=blue]{hyperref}
\usepackage{mathrsfs}
\usepackage{hyperref}
\hypersetup{colorlinks=true,urlcolor=blue,citecolor=blue,linkcolor=blue}
\usepackage{courier}
\usepackage{tikz}
\usetikzlibrary{calc,matrix,arrows,decorations.markings}
\usepackage{array}
\usepackage{verbatim}
\usepackage{color}
\usepackage{enumerate}
\usepackage{nicefrac}
\usepackage{listings}

\advance \topmargin by -\headheight
\advance \topmargin by -\headsep

\evensidemargin \oddsidemargin
\newcommand{\cal}{\mathcal}

\newcommand{\R}{\mathbb{R}}
\newcommand{\C}{\mathbb{C}}
\newcommand{\Z}{\mathbb{Z}}

\newcommand{\dd}{\,d}

\newcommand{\Tr}{\operatorname{Tr}}

\newcommand{\N}{{\rm N}}
\newcommand{\re}{\operatorname{re}}

\newtheorem{theorem}{Theorem}[section]
\newtheorem{prop}[theorem]{Proposition}

\newtheorem{lemma}[theorem]{Lemma}

\newtheorem*{theorem*}{Theorem}
\newtheorem*{conj*}{Conjecture}
\newtheorem*{lemma*}{Lemma}
\newtheorem*{prop*}{Proposition}
\newtheorem*{cor*}{Corollary}

\numberwithin{equation}{section}

\begin{document}

\title{Selberg Integral over Local Fields}

\date{\today}

\author[Zenan Fu]{Zenan Fu}

\address{School of Mathematical Science, Zhejiang University, Hangzhou 310027, Zhejiang, P.R. China}

\email{zenanfu@zju.edu.cn }

\author[Yongchang Zhu]{Yongchang Zhu}

\address{Department of Mathematics, The Hong Kong University of Science and Technology, Clear Water Bay, Kowloon, Hong Kong}

\email{mazhu@ust.hk}

	\maketitle
	\section{\large Introduction }\label{section1}
	
 Selberg introduced his beautiful integral formula in 1944 \cite{Sel} that asserts
\begin{eqnarray}\label{selberg}
   S_{n}(a ,b ,c )&=& \int _{0}^{1}\cdots \int _{0}^{1}\prod _{i=1}^{n}t_{i}^{a -1}(1-t_{i})^{b -1}
   \prod _{1\leq i<j\leq n}|t_{i}-t_{j}|^{2c }\,dt_{1}\cdots dt_{n} \nonumber \\
    &=& \prod _{j=0}^{n-1}{\frac {\Gamma (a +jc )\Gamma (b +jc )\Gamma (1+(j+1)c )}
    {\Gamma (a +b +(n+j-1)c )\Gamma (1+c )}}
  \end{eqnarray}
  where $n$ is a positive integer and $a, b , c $ are complex numbers satisfying
   $ {\rm re} \, a > 0 $,  $ {\rm re} \, b > 0 $ and
   $ {\rm re} \, c > -{\rm min} \, \{ 1/ n, {\rm re} \, a / ( n-1), {\rm re} \, b / ( n-1) \}$.
We refer to Forrester and Warnaar's exposition \cite{FW} for the history, generalizations and the applications of Selberg integral.
Evans \cite{E1} conjectured a finite field analog of Selberg integral formula in 1980. And  Anderson \cite{An}
 proved a major case of it in 1981 and his ideas  was used to obtained the complete result \cite{E2}.
  On the other hand, Aomoto \cite{Ao} proved an analog of Selberg integral for complex field ${\mathbb{C}}$ in 1987.
 The purpose of the present paper is to formulate and prove Selberg integral formula for local fields of characteristic zero.
 To state our results, we first introduce some notations.

    Let $F$ be a local field of characteristic zero,  $\psi $ be a non-trivial additive character, $ dx $ be the self-dual Haar measure with respect to  $\psi $, i.e.,
	the Fourier transform defined  using $dx$,
	\begin{equation}\label{1.2}  {\cal F} f ( x ) = \int_F f ( y ) \psi ( x y ) d y  \end{equation}
	is an isometry. The absolute value $|a |_F$ of $a \in F$  is defined by the formula ${\rm vol} ( a U ) = | a|_F \, {\rm vol} (  U ) $.
	For a quasi-character $c $ of $F^*$, its real part, denoted by ${\rm  re } \, c $, is the unique real number satisfying the
    condition
	\begin{equation}   | c ( x ) |  = | x |_F^{{\rm re}\, c} .\end{equation}
	The $\rho $-factor $\rho ( c ) $ of  $c$ is defined as
	\begin{equation}\label{1.3}
	\rho ( c ) := \frac { \int f ( x ) c (x ) | x|_F^{-1} dx } { \int \mathcal{F} f ( x )  c^{-1} ( x ) d x } . \end{equation}
	where $f$ is a Schwartz function on $F$, the both integrals converge when $0< {\rm re } \, c <1 $ and
  they are understood as analytic continuation for general $c$,  see \cite{T}. We also use the gamma function $\Gamma (c)$
    as in \cite{GGPS}, it is related to $\rho ( c)$ by $ \Gamma ( c ) = c ( -1) \rho (c ) $.
	We denote  $\chi_0 (x ) =  | x|_F$.

 The choice of $\psi $ induces an additive character $\psi_E$ on any finite extension $E$ by $\psi_ E ( x ) = \psi ( {\rm Tr} \, x ) $
 and therefore the self-dual Haar measure $d_E $ on $E$.  Then for any quasi-character $c$ of $E^*$, we can similarly define the $\rho_E ( c ) $ and $\Gamma_E ( c )$.
  The absolute value of $E$ is denoted by $ |x|_E$, and we have $|x|_E=|\N(x)|_F$.
 Note that for a quasi-character $c$ of $F^*$, its composition $ \N \circ c $ with the norm map has the property
  $ {\rm re} (  \N \circ c ) = {\rm re} \, c $.

  Let $M_n$ be the space $\{f\in F[x] \ |\ f\  \text{is monic of degree}\  n\}$ for $n\in\Z_{>0}$. Equal $M_n$ with $F^{n}$ via the map
  \begin{equation}
  \eta : M_n \rightarrow F^{n}, f(x)= x^{n}+b_{n-1}x^{n-1}+\cdots+b_0 \mapsto (b_{n-1},\cdots,b_0).
  \end{equation}
   Let $\dd f$ be the measure inherits from the product Haar measure of $F^{n}$.
	 Recall the  discriminant $\Delta (f ) $ of $f \in M_n$
      is defined as
       \[  \Delta ( f ) = (-1)^{ \frac 12 n (n -1 ) }\prod_{i\neq j}(\alpha_i-\alpha_j) =  \prod_{1\leq i < j\leq n }(\alpha_i-\alpha_j)^2  ,\]
         where $\alpha_i$ ($i=1, \dots , n$)
       are the roots of $f$. Empty product is considered to be equal to $1$.  For an irreducible polynomial $ h (x )$ over $F$,
        we denote the field $ F [ x] / ( h( x ) )$ by $F_h $ and if $\chi $ is a quasi-character of $F^*$, we denote
         $ \Gamma_h ( \chi )$ the Gamma function  $ \Gamma_{ F_h } ( \chi \circ \N ) $.

	For quasi-characters $\alpha,\beta,\gamma$ of $F^*$ in the region $R_n$ given by
	 \begin{equation}\label{region}
	 \begin{split}
	 &\re \alpha,\re\beta,\re\gamma>0;\\
	  &\re \alpha+\re\beta+2(n-1)\re\gamma<1,
	  \end{split}
	 \end{equation}
	we define the Selberg integral over $F$ as
	\begin{equation}\label{sel}
	S_n(\alpha,\beta,\gamma)
	=\alpha^n\gamma^{\frac{n(n-1)}{2}}(-1) \int_{f\in M_{n}}\alpha\chi_0^{-1}(f(0))\beta\chi_0^{-1}(f(1))\gamma\chi_0^{-\frac1{2}}(\Delta(f))\prod_{i=1}^l\frac{\Gamma_{h_i}(\gamma)}{\Gamma(\gamma)^{\deg h_i}}\dd f
	\end{equation}
	where we write $f(x)=\prod_{i=1}^l h_i(x)$ with $h_i$ monic irreducible polynomial over $F$.
Note that polynomials with zero discriminant have measure zero, so we may assume $f$ has no repeated roots.

	We have the following theorem that generalizes Selberg integral formula (\ref{selberg}),
	\begin{theorem}\label{thm}   The integral $S_n(\alpha,\beta,\gamma)$  converges for $ ( \alpha , \beta , \gamma ) \in R_n $ and
            	\begin{equation}\label{formula}
		 S_n(\alpha,\beta,\gamma)=\prod_{j=0}^{n-1}\frac{\Gamma(\alpha\gamma^{j})\Gamma(\beta\gamma^{j})\Gamma(\gamma^{j+1})}{\Gamma(\alpha\beta\gamma^{n+j-1})\Gamma(\gamma)}.
		\end{equation}
	\end{theorem}

  In the case $F$ is a finite field, we consider the Gauss sums for $F$ as an analog of the Gamma factors, then the factor $ \prod_{i=1}^l\frac{\Gamma_{h_i}(\gamma)}{\Gamma(\gamma)^{\deg h_i}}$ is equal to $1$
   by the Hasse-Davenport relation, so this term does not appear in the finite field generalization of Selberg integral in  \cite{E1} \cite{An} (see also
        \cite{AAR}). For the case $F={\mathbb C}$, since all $ F_h = {\mathbb C}$, the above factor is also equal to $1$, (\ref{formula}) reduces to Aomoto's generalization
         of Selberg integral \cite{Ao}, see the end of Section \ref{section3} for more detail.
        While Aomoto \cite{Ao} considers (\ref{sel}) for $F={\mathbb{C}}$ as a pairing in certain twisted de Rham cohomology and homology, our integral is just the ordinary integral.
         Our domain of convergence (\ref{region}) for the case $F={\mathbb{C}}$ is contained in Aomoto's defining domain for (\ref{sel}). 
          Note that only the unramified quasi-characters $\alpha , \beta , \gamma $ are considered in \cite{Ao}.
    We also remark that for general $F$, $ f \mapsto \prod_{i=1}^l\frac{\Gamma_{h_i}(\gamma)}{\Gamma(\gamma)^{\deg h_i}}$
     is a locally constant function on the region
      $ M_n -\{  f \, | \, \Delta ( f ) = 0  \} $,  which is a disjoint unions of open sets
       of $f$'s with  $ F [ x ] / ( f ( x ) ) $ isomorphic to a direct product of extensions of $F$ of fixed types.

    We like to comment that (\ref{formula}) for the case $F={\mathbb R}$ is not equivalent to the original Selberg integral (\ref{selberg}), which
     should be understood as (\ref{formula}) for $({\mathbb R}_{\geq 0 } , + , \cdot ) $.

     We prove Theorem \ref{thm} by   evaluating a double integral in two different ways, which
      give a recursive formula  relating $S_n$ to $S_{n-1} $. This method is due to Anderson \cite{An} \cite{An2}.

    This paper is organized as follows. In Section \ref{section2}, we prove an extension of beta integral and two corollaries that are used in the
     proof of Theorem \ref{thm}.  In Section \ref{section3}, we prove our main theorem and compare our formula in the case $F={\mathbb{C}}$ with Aomoto's.

\

	\section{\large  Generalized Beta Integrals. }\label{section2}

	\begin{lemma}\label{lemma1}
		Let $ V$ be a finite dimensional vector space over $F$, $d x $ be a Haar measure on $V$ as an additive group.  If
		$ N : V \to  {\R}_{\geq 0 } $ is a $F$-norm, then for ${\rm re } \, s <  - {\rm dim}_F V $,
		\begin{equation}
		\int_{N( x )\geq r }    N ( x )^s  d x
		\end{equation}
		converges for any $ r > 0 $.
	\end{lemma}

	\begin{proof}   The cases $F={\mathbb R}, {\mathbb C}$ are standard exercises in calculus. We assume $F$ non-Archimedean. Let $ n = {\rm dim}_F V $.
	Since all the norms on $V$ are equivalent, we may assume that $V$ is a field extension of $F$ and $ N ( x ) = | x|_V^{ \frac 1 n } $.
	The result follows from the fact that $ \int_{ | x | > r } | x |_V^{\frac s n } d x $
	converges when $  \frac 1 n  {\rm re } \, s < -1 $.
	\end{proof}

It is known that $\rho ( c )$ for ${\rm re} \, c > 0 $  can be written as an integral
	\begin{equation}
	\rho ( c ) = c(-1)   \int_{ F }  \psi ( x ) c ( x ) | x |_F^{-1} dx  :=   c(-1) \lim_{ r \to \infty}   \int_{  |x| \leq r }  \psi ( x ) c ( x ) | x |_F^{-1} dx .
	\end{equation}
 Equivalently, the Gamma function of $c$ can be written as
	\begin{equation}\label{2.3}
	\Gamma(c)=c(-1)\rho(c)=\int_{ F }  \psi ( x ) c ( x ) | x |_F^{-1} dx,
	\end{equation}
  see \cite{GGPS}.  This formula is used to prove the beta integral formula \cite{GGPS}:
  \begin{equation}\label{2.4}
      \int_{ F} c_1 ( x ) | x|^{-1} c_2 ( 1-x ) | 1- x |^{-1} d x = \frac{ \Gamma (c_1 ) \Gamma (c_2 ) } { \Gamma (c_1 c_2 ) }.
  \end{equation}
  where the convergence region is $ {\rm re} \, c_1 > 0 , {\rm re } \, c_2 > 0 , {\rm re} \, c_1 + {\rm re} \, c_2 < 1 $.
   See \cite{Tai} page 61 for a detailed proof.  We will need the following generalization of (\ref{2.4}).
	
	\begin{lemma}\label{lemma2}   Let $  E_1 , \dots , E_k $ be finite extensions of $F$ of degrees $d_1 , \dots , d_k $ with $ d = d_1 + \dots + d_k \geq 2 $,
		$ {\rm Tr}_i : E_i \to F$ be the trace map,  $ \phi :  E_1 \times \dots \times E_k \to F$ be the $F$-linear map
		given by
		\begin{equation} \phi ( x_1 , \dots , x_k ) =  \sum_{ i =1 }^k  {\rm Tr}_i (  x_ i ). \end{equation}
		Let $ S= \phi^{-1} ( 1) $ and $ds $ be the measure on $S$ uniquely determined by the conditions that it is invariant under translations by $\phi^{-1} ( 0 )$
        and that the map
         \begin{equation} F^* \times S \to   E_1 \times \cdots \times E_k ,      ( a  , ( x_1, \dots , x_k ) ) \mapsto ( a x_1 , \dots , a x_k ) \end{equation}
         changes the measure $ |a|^{d-1} d a \, d s $ to $ dy_1 \cdots d y_k $ (where $da, dy_i $ are self-dual measures
          defined in Section \ref{section1}).
        	If
        $c_1 , \dots , c_k$ are quasi-characters on $E_1^*, \dots , E_k^*$ with
		$ {\rm re} \, c_i > 0  $ for all $i$ and
		\begin{equation}  d_1 {\rm re} \, c_1 + \dots + d_k {\rm re } \, c_k < 1.   \end{equation}
		then the integral
		\begin{equation}\label{2.5}
		\int_{  S }   c_1 (x_1 ) |x_1|_{E_1}^{-1}   \cdots c_k ( x_k ) |x_k|_{E_k}^{-1} d s
		\end{equation}
		converges  and is equal to
		\begin{equation}\label{2.6}
		\frac { \prod_{ i=1}^k \Gamma_{ E_i} ( c_i ) } { \Gamma ( c ) }
		\end{equation}
		where $c= (c_1|_{F^*} )\cdots (c_k|_{F^*})$.
	\end{lemma}

	\begin{proof}   We prove first the convergence by induction.  We may assume all $ c_i $ are $ {\R}_{>0}$ valued,
	so $ c_i ( x) = |x|_{E_i}^{r_i} , r_i \in {\R}$.
	If $k=1 $, then $ d_1 > 1 $, let $  v \in E_1 $ satisfy $\phi ( v ) =1 $,
	the integral is
	\begin{equation} \int_{ S }   c_1 (x_1 ) |x_1|_{E_1}^{-1} d s =   \int_{ \phi^{-1} ( 0 ) }  | y + v |_{E_1}^{r_1 -1} d y \end{equation}
	where we change the variable $ x_1 \to y + v $ and $d y $ is the Haar measure on $\phi^{-1} ( 0 ) $ induced from $ds$ on $ S$.
	Since $ y + v $ is never $0$, the integrand has no finite singular point, it is enough to prove
	\begin{equation}   \int_{ y: y \in \phi^{-1} ( 0 ), | y | \geq r }   | y + v |_{E_1}^{r_1 -1} d y < \infty \end{equation}
	for any $r>0$. Note
	\begin{equation}   | y + v |_{E_1}^{r_1 -1}   <  C  | y  |_{E_1}^{{\rm re} \, c_1 -1} \end{equation}
	for $ y$ with $ |y|_{E_1} > r $ ($C$ is a constant depending on $r$). So it is enough to prove
	\begin{equation} \int_{  | y |_{E_1} \geq r }    | y  |_{E_1}^{ r_1  -1} < \infty. \end{equation}
	Notice that $ y \mapsto  | y  |_{E_1}^{\frac 1 { d_1 } } $ is an $F$-norm. The result follows from Lemma \ref{lemma1} for $V= \phi^{-1} ( 0 )$.
	
	The case $k=2$ and $ d_1=d_2 =1 $ is (\ref{2.4}).
	The case  $ k=2 $ and $d_1 + d_2 > 2 $ can be proved using the induction on $d_1 + d_2 $.
	For $ k \geq 3 $, we have
	\begin{equation}
	\begin{split}
        &\int_S  c_1 ( x_1 ) |x_1|_{E_1}^{-1} \cdots c_k ( x_k  )|x_k|_{E_k}^{-1} d x  \\
       =&    \int_{ {\rm Tr}\,  x_1 + {\rm Tr}\, x_2 = 1 } c_1 ( x_1 ) |x_1|_{E_1}^{-1} c_2 ( x_2 ) |x_2|_{E_2}^{-1}  d ( x_1 , x_2 ) \\
        &\cdot  \int_{ a +{\rm Tr}\, x_3 + \dots + {\rm Tr}\, x_k  =1 }   | a |^{ d_1 r_1 + d_2 r_2 -1 } c_3 ( x_3 ) |x_3|_{E_3}^{-1}  \cdots c_k ( x _ k ) |x_k|_{E_k}^{-1} d (a , x_3 , \dots , x_k ).
	\end{split}		
	\end{equation}
	where $d ( x_2 , x_3)$ is  certain measure on $ {\rm Tr}\,  x_1 + {\rm Tr}\, x_2 = 1 $ invariant under the translations by $ {\rm Tr}\,  x_1 + {\rm Tr}\, x_2 = 0 $
  and
$ d (a , x_3 , \dots , x_k )$ is certain measure on $  a +{\rm Tr}\, x_3 + \dots + {\rm Tr}\, x_k  =1 $ invariant under the translations by
  $  a +{\rm Tr}\, x_3 + \dots + {\rm Tr}\, x_k  = 0 $.
Both integrals on the right converge by induction assumption.
	Finally, we have
	\begin{equation}
	\begin{split}
        & \Gamma ( c )  \int_{  S }   c_1 (x_1 ) |x_1|_{E_1}^{-1}   \cdots c_k ( x_k ) |x_k|_{E_k}^{-1} d s \\
		= &\int_F \psi ( a ) c ( a ) |a|^{-1 } da   \int_{  S }   c_1 (x_1 ) |x_1|_{E_1}^{-1}   \cdots c_k ( x_k ) |x_k|_{E_k}^{-1} d s \\
		=&    \int_{ F\times  S }  \psi ( {\rm Tr} (ax_1 + \dots + ax_k)  )   c_1 ( ax_1 ) |a x_1|_{E_1}^{-1}   \cdots c_k (a x_k ) |a x_k|_{E_k}^{-1}  |a|^{ d -1 } d s da \\
		=& \int_{ E_1 \times \dots \times E_n}  \psi ( {\rm Tr} ( y_1 ) +  \dots + {\rm Tr} ( y_k  ) )c_1 (y_1 ) |y_1|_{E_1}^{-1}   \cdots c_k ( y_k ) |y_k|_{E_k}^{-1} \textstyle\prod d y_i \\
		=& \prod_{ i=1}^k \Gamma_{E_i} ( c_i ).
	\end{split}
	\end{equation}
	\end{proof}
	
  We remark that if the map $\phi $ in Lemma \ref{lemma2} is replaced by $ \phi ( x_1 , \dots , x_k ) =  \sum_{ i =1 }^k  {\rm Tr}_i ( a_i x_ i )$ for $a_i \in E_i^*$,
   the integral (\ref{2.5}) is convergent under the same conditions on $ {\rm re} \, c_i $'s and the result
    is (\ref{2.6}) times $ \prod_{i=1}^k c_i ( a_i^{-1} )$. This can be proved using Lemma \ref{lemma2} and the change of variable $ a_i x_i \mapsto x_i $.

	For any $g\in F[x]$, let $n=\deg g$, and denote $F[x]/(g(x))$ by $F_g$ and equal $F_g$ with $F^{n}$ via the map
	\begin{equation}
	 \eta' : F_g \rightarrow F^{n}, f(x)= b_{n-1}x^{n-1}+b_{n-2}x^{n-2}+\cdots+b_0 \mapsto (b_{n-1},b_{n-2},\cdots,b_0).
	 \end{equation}
	Let $\dd_g f$
	be the measure inherits from the product haar measure of $F^{n}$.
	
	Let $G(x)$ be a monic separable polynomial over $F$. Assume $G(x) = \prod_{i=1}^k g_i(x)$, with $g_i(x)$ different monic irreducible polynomials over $F$. Then we have an isomorphism: $\varphi:\  E = F[x]/(G(x)) \rightarrow \prod_{i=1}^k F[x]/(g_i(x))$, such that  \begin{equation} \varphi(f)=(\varphi_1(f),\varphi_2(f),\cdots,\varphi_k(f))= (f\text{ mod } {g_1},f\text{ mod } {g_2},\cdots,f\text{ mod } {g_k}).\end{equation}

     Let $g_i(x) = \prod_{j=1}^{d_i}(x-\alpha_{ij})$ with $\alpha_{ij}\in\overline{F}$, $d_i = \deg g_i$, $1\le i\le k$. Define the Trace and Norm maps on  $F_{g_i}$,
	\begin{equation}
	\Tr_{g_i}(f) := \sum_{j=1}^{d_i}f(\alpha_{ij})\quad \text{and}\quad \N_{g_i}(f) := \prod_{j=1}^{d_i}f(\alpha_{ij}).
	\end{equation}
 They are just usual trace and norm for field extension $ F_{ g_i } $ over $F$.
 Let $\psi_i : F_{g_i} \to S^1 $ be the additive character  $\psi_i (f) = \psi(\Tr_{g_i} f ) $, it defines
  the Fourier transform on $\mathcal{S}(F_{g_i})$ as ${\cal F} h(y) := \int_{F_{g_i}}h(x)\psi_i (xy)\dd_i x$, where $\dd_i x$ is the unique measure on $F_{g_i}$ such that
   ${\cal F}$ is an isometry.  We need to know the relations of Haar measures $ \dd_i x $ and $ \dd_{g_i}x $. For this purpose, we prove

   \begin{lemma}\label{lemma2.3}  Let $  D$ be a $n\times n$ non-degenerate symmetric matrix over $F$,  and $ d_D x $ be the unique Haar measure on $F^n$ such that
    the Fourier transform
    \begin{equation}
     {\cal F}_D f ( y ) = \int_{F^n } f ( x ) \psi \left(  x^T D y  \right)  d_D x
     \end{equation}
   is an isometry.  Then $  d_D x = | {\rm det} \, D|_F^{\frac 12 } d x $, where $ dx = d x_1 \cdots d x_n $ is the product measure of the self-dual measure on $F$ determined by $\psi$.
   \end{lemma}

 \noindent {\it Proof.}  By the uniqueness of the Haar measure, we have $ d_D x = C d x $ for some positive scalar $C$.
   We first prove the case $F$ is non-Archimedean.  Let $R$ be the ring of integers in $F$, $\pi \in R$ be a local parameter, $q = | R/ \pi R |$. There exists $\delta \in {\mathbb Z}$ such that
    $ \psi ( \pi^{-\delta } R ) = 1$ and $ \psi ( \pi^{-\delta -1} R ) \ne 1$. Then
      $  {\cal F} 1_R ( x ) =   q^{ - \frac \delta 2}   1_{ \pi^{-\delta } R } ( y )  $, where $ 1_S$ denotes the characteristic function of set $S$, ${\cal F}$ is as in (\ref{1.2}).
   It is easy to see that
   \begin{equation}
      {\cal F}_D  1_{ R^n } ( x )   =   C  q^{ - \frac { n \delta } 2}     1_{   \pi^{-\delta } D^{-1} R^n } ( x )
      \end{equation}
   The condition that  that  ${\cal F}_D$ is an isometry implies that $ C = | {\rm det} \, D|_F^{\frac 12 } $. For $F={\mathbb R}$ or ${\mathbb C}$ we use Gaussian functions instead of $1_{ R^n}$ to get the result.
   \hfill $\Box $

\

    From the Lemma, it's easy to see that $\dd_i x = |\Delta(g_i)|_F^{\frac1{2}}\dd_{g_i}x$.
      For any quasi-character $\chi$ on $F^*$, $ \chi \circ \N $ is a quasi-character of $F_{ g_i}^*$, we write   $\Gamma_{E}(\chi\circ\N)=\Gamma_{g_i}(\chi).$
      The $F_{g_i}$-version of (\ref{2.3}) reads as
      \begin{equation}
       \Gamma_{g_i}(\chi) = \int_{F_{g_i}}\psi(\Tr_{g_i}(x))\chi(\N_{g_i}(x))|x|_{F_{g_i}}^{-1}\dd_i x,
       \end{equation}
       where $|x|_{F_{g_i}}$ is the absolute value on the field $F_{g_i}$ and we have $|x|_{F_{g_i}}=|\N_{g_i}(x)|_F$.
       \
	
	Similarlly, define Trace and Norm map on $F_{G}$ as
	\begin{equation}
	\Tr_G(f) := \sum_{i=1}^{k}\Tr_{g_i}(\varphi_i(f))\quad \text{and}\quad \N_G(f) := \prod_{i=1}^{k}\N_{g_i}(\varphi_i(f)).
	\end{equation}
	
	Let $\psi : F_{G} \to S^1 $ be the additive character  $\phi (f) = \psi(\Tr_{G} f ) $, it defines the Fourier transform on $\mathcal{S}(F_{G})$ as ${\cal F}_G h (y) := \int_{F_{G}}h(x)\phi(xy)\dd_* x$, where
 $\dd_* x$ is the Haar measure such that ${\cal F}_G$ is an isometry.
   Using Lemma \ref{lemma2.3},  we can prove that
    \begin{equation}\label{ch} \dd_* x = |\Delta(G)|_F^{\frac1{2}}\dd_{G}x.\end{equation}
   It is also easy to see that $ d_*x=\prod_{i=1}^k\dd_i \varphi_i(x) $.
	
For $f(x),g(x)\in F[x]$, assume $f(x)=a\prod_{i=1}^n(x-\alpha_i)$ and $g(x)=b\prod_{j=1}^m(x-\beta_j)$, where $a,b\in F$ and $\alpha_i,\beta_j\in\overline{F}, 1\le i\le n,1\le j\le m$,  the resultant of $f$ and $g$ is defined
  as
	\begin{equation}
	R(f,g):= a^mb^n\prod_{i=1}^n\prod_{j=1}^m(\alpha_i-\beta_j)=a^m\prod_{i=1}^ng(\alpha_i)=(-1)^{mn}b^n \prod_{j=1}^mf(\beta_j).
	\end{equation}
     In particular, for $  g = b \in F^* $, we have $ R ( f , b ) = b^{ {\rm deg } \, f } $.

     We will use the following properties of resultant, which can be proved by definition:
     \begin{equation}
     \begin{split}
     R(f,f')&=(-1)^{\frac12n(n-1)}\Delta(f);\\
     R(f,g_1g_2)&=R(f,g_1)R(f,g_2);\\
     R(f_1f_2,g)&=R(f_1,g)R(f_2,g).
     \end{split}
     \end{equation}

	Then we have the following two propositions:
	
	\begin{prop}\label{prop1}
		For any $G(x)\in F[x]$ such that $G(x) = \prod_{i=1}^k g_i(x)$, with $g_i(x)$ different monic irreducible polynomial over $F$ such that $g_i(x) = \prod_{j=1}^{d_i}(x-\alpha_{ij})$ with $\alpha_{ij}\in\overline{F}$, $d_i = \deg g_i$, $1\le i\le k$ and $n=\sum_{i=1}^k d_i$. Then we have
		\begin{equation}\label{1}
		\int_{f\in M_{n-1}}\chi\chi_0^{-1}(R(G,f))\dd f = \chi_0(\Delta(G))^{-\frac1{2}}\chi(R(G,G'))\frac{\prod_{i=1}^{k}\Gamma_{g_i}(\chi)}{\Gamma(\chi^n)}.
		\end{equation}
		where $\chi_0(x)=|x|_F$, and $\chi$ is a quasi-character of $F^*$ such that $0<n\re\chi<1$.
	\end{prop}
	
	\begin{proof}
		View $M_{n-1}$ as a subset of $F_G$. For any $f\in F_G$, we have
		\begin{equation}
		f(x)=\sum_{i=1}^k\sum_{j=1}^{d_i}f(\alpha_{ij})\frac{\prod_{(k,l)\neq (i,j)}(x-\alpha_{kl})}{\prod_{(k,l)\neq (i,j)}(\alpha_{ij}-\alpha_{kl})}
		\end{equation}
		by Lagrange interpolation formula. Hence
		\begin{equation}
		f \in M_{n-1}
		\Longleftrightarrow
		 \sum_{i=1}^k\sum_{j=1}^{d_i}\frac{f(\alpha_{ij})}{G'(\alpha_{ij})}=\sum_{i=1}^k\Tr_{g_i}\left(\frac{\varphi_i(f)}{\varphi_i(G')}\right)=1.
		\end{equation}
		Let $\phi(g)= \sum_{i=1}^k\Tr_{g_i}\left(\varphi_i(g)/\varphi_i(G')\right)$, then $M_{n-1}=\phi^{-1}(1)$. And we have
		\begin{equation}
	\int_{f\in M_{n-1}}\chi\chi_0^{-1}(R(G,f))\dd f = \int_{f\in M_{n-1}} \prod_{i=1}^k \chi ( \N_{g_i} \varphi_i ( f ) ) | \varphi_i ( f )|_{ F_{g_i}}^{-1}  \dd f.
     \end{equation}
   By Lemma \ref{lemma2} and the remark after its proof, we see the integral converges 	in the region $0< n\re\chi<1$.	
    At this point, we may change the variable $ \varphi_i (f ) /  \varphi_i ( G') \mapsto x_i$ and use Lemma \ref{lemma2} to prove (\ref{1}), but it is messy to compute the change of
     various Haar measures, we choose to proceed directly as follows.
      We have an one-to-one map $\delta:\ F \times M_{n-1}\rightarrow F^n  $, such that for $a\in F$, $f(x)=x^{n-1}+b_{n-2}x^{n-2}+\cdots+b_0\in M_{n-1}$, $\delta(a,f)=h= af(x)=ax^{n-1}+ab_{n-2}x^{n-2}+\cdots+ab_0=(a,ab_{n-2},\cdots,ab_0)$. The Jacobian of $\delta$ is equal to $a^{n-1}$. Note that $F^n\setminus\delta(M_{n-1}\times F)$ has measure zero in $F^n$. So we have \begin{equation}
		\begin{split}
        & \Gamma ( \chi^n ) \int_{f\in M_{n-1}}\chi\chi_0^{-1}(R(G,f))\dd f \\
		=& \int_F\psi(a)\chi(R(G,a))|a|_F^{-1}\dd a \int_{f\in M_{n-1}}\chi\chi_0^{-1}(R(G,f))\dd f \\
		= & \int_{F_G}\psi(a)\chi\chi_0^{-1}(R(G,h))\dd_G h\quad (\text{where $a$ is the highest coefficience of the polynomial $h$.})\\
		= &\chi_0(\Delta(G))^{-\frac1{2}}\int_{F_G}\psi(a)\chi\chi_0^{-1}(R(G,h))\dd_* h\\
		= &\chi_0(\Delta(G))^{-\frac1{2}}\int_{\prod_{i=1}^kF_{g_i}}\psi\left( \sum_{i=1}^k\sum_{j=1}^{d_i}\frac{h(\alpha_{ij})}{G'(\alpha_{ij})}\right) \chi\chi_0^{-1}\left( \prod_{i=1}^k\prod_{j=1}^{d_i}h(\alpha_{ij})\right) \textstyle\prod\limits_{i=1}^k\dd_{i} \varphi_i(h)\\
		= &\chi_0(\Delta(G))^{-\frac1{2}}\chi\chi_0^{-1}\left( \prod_{i=1}^k\prod_{j=1}^{d_i}G'(\alpha_{ij})\right) \int_{\prod\limits_{i=1}^kF_{g_i}}\prod_{i=1}^k\psi\left(\Tr_{g_i}\frac{\varphi_i(h)}{\varphi_i(G')}\right)\prod_{i=1}^k\chi\chi_0^{-1}\left(\N_{g_i}\frac{\varphi_i(h)}{\varphi_i(G')}\right)\textstyle\prod\limits_{i=1}^k\dd_{i} \varphi_i(h)\\
		=&\chi_0(\Delta(G))^{-\frac1{2}}\chi\chi_0^{-1}(R(G,G'))\prod_{i=1}^k\left|\varphi_i(G')\right|_{F_{g_i}}\prod_{i=1}^{k}\Gamma_{g_i}(\chi)\\
		=&\chi_0(\Delta(G))^{-\frac1{2}}\chi\chi_0^{-1}(R(G,G'))\left|\prod_{i=1}^k\prod_{j=1}^{d_i}G'(\alpha_{ij})\right|_F\prod_{i=1}^{k}\Gamma_{g_i}(\chi)\\
		=&\chi_0(\Delta(G))^{\frac1{2}}\chi\chi_0^{-1}(R(G,G'))\prod_{i=1}^{k}\Gamma_{g_i}(\chi)\\
		=&\chi_0(\Delta(G))^{-\frac1{2}}\chi(R(G,G'))\prod_{i=1}^{k}\Gamma_{g_i}(\chi)
		\end{split}
		\end{equation}
where we used the relation $ \dd_G h = \chi_0 ( \Delta ( G) )^{-\frac 12 } \dd_* h $ in the third equation, which is just (\ref{ch}), a corollary of Lemma \ref{lemma2.3};
 and we used Lagrange interpolation in the fourth equation:
 \begin{equation}
 h(x)=\sum_{i=1}^k\sum_{j=1}^{d_i}h(\alpha_{ij})\frac{\prod_{(k,l)\neq (i,j)}(x-\alpha_{kl})}{\prod_{(k,l)\neq (i,j)}(\alpha_{ij}-\alpha_{kl})}
 \Rightarrow
 a=\sum_{i=1}^k\sum_{j=1}^{d_i}\frac{h(\alpha_{ij})}{G'(\alpha_{ij})}
 \end{equation}
		and made change of variables in the sixth equation:
		\begin{equation}
		\varphi_i(h)\mapsto \frac{\varphi_i(h)}{\varphi_i(G')},\quad 1\le i\le k.
		\end{equation}
		\end{proof}

    \begin{prop}\label{prop2}
    	For any $G(x)\in F[x]$ such that $G(x) = \prod_{i=1}^k g_i(x)$, with $g_i(x)$ different monic irreducible polynomial over $F$ such that $g_i(x) = \prod_{j=1}^{d_i}(x-\alpha_{ij})$ with $\alpha_{ij}\in\overline{F}$, $d_i = \deg g_i$, $1\le i\le k$ and $\sum_{i=1}^k d_i=n-1$. Assume $G(0)G(1)\neq 0$, let $S=x(x-1)G$ and let $\alpha,\beta,\gamma$ be quasi-characters of $F^*$ such that $\re \alpha>0,\re \beta >0,\re \gamma>0$ and $\re\alpha+\re\beta+(n-1)\re \gamma<1$, then we have:
    	\begin{equation}\label{2}
    	\begin{split}
    	&\int_{f\in M_{n}}\alpha\chi_0^{-1}(f(0))\beta\chi_0^{-1}(f(1))\gamma\chi_0^{-1}(R(G,f))\dd f \\
    	&= \alpha(-1)\chi_0(\Delta(G))^{-\frac1{2}}\alpha\gamma\chi_0^{-1}(G(0))\beta\gamma\chi_0^{-1}(G(1))\gamma(R(G,G'))\frac{\Gamma(\alpha)\Gamma(\beta)\prod_{i=1}^{k}\Gamma_{g_i}(\gamma)}{\Gamma(\alpha\beta\gamma^{n-1})}.
    	\end{split}    	
    	\end{equation}
    \end{prop}
    \begin{proof}
        The proof is similar.
    	View $M_n$ as a subset of $F_S$. For any $f\in F_S$, we have
    	\begin{equation}
    	f(x)=\sum_{i=1}^k\sum_{j=1}^{d_i}f(\alpha_{ij})\frac{x(x-1)\prod_{(k,l)\neq (i,j)}(x-\alpha_{kl})}{\alpha_{ij}(\alpha_{ij}-1)\prod_{(k,l)\neq (i,j)}(\alpha_{ij}-\alpha_{kl})}+f(0)\frac{(x-1)G(x)}{S'(0)}+f(1)\frac{xG(x)}{S'(1)}.
    	\end{equation}
    	by Lagrange interpolation formula. Hence
    	\begin{equation}
    	f \in M_{n}
    	\Longleftrightarrow
    	\sum_{i=1}^k\sum_{j=1}^{d_i}\frac{f(\alpha_{ij})}{S'(\alpha_{ij})}+\frac{f(0)}{S'(0)}+\frac{f(1)}{S'(1)}= \sum_{i=1}^k\Tr_{g_i}\left(\frac{\varphi_i(f)}{\varphi_i(S')}\right)+\frac{f(0)}{S'(0)}+\frac{f(1)}{S'(1)}=1.
    	\end{equation}
    	Let $\phi(f)= \sum_{i=1}^k\Tr_{g_i}\left(\varphi_i(f)/\varphi_i(S')\right)+f(0)/S'(0)+f(1)/S'(1)$, then $M_n=\phi^{-1}(1)$. And we have
    	\begin{equation}
    	\begin{split}
    	&\int_{f\in M_{n}}\alpha\chi_0^{-1}(f(0))\beta\chi_0^{-1}(f(1))\gamma\chi_0^{-1}(R(G,f))\dd f \\
    	=& \int_{f\in M_{n}} \alpha(f(0))|f(0)|_F^{-1}\beta(f(1))|f(1)|_{F}^{-1}\prod_{i=1}^k \chi ( \N_{g_i} \varphi_i ( f ) ) | \varphi_i ( f )|_{ F_{g_i}}^{-1}  \dd f
    	\end{split}
    	\end{equation}
    	By Lemma \ref{lemma2} and the remark after its proof, we see the integral converges in the region: $\re \alpha,\re \beta,\re \gamma>0;\, \re\alpha+\re\beta+(n-1)\re \gamma<1$. Similar to Proposition \ref{prop1}, we compute the integral as follows:
    	\begin{equation}
    	\begin{split}
    	&\Gamma(\alpha\beta\gamma^{n-1})\int_{f\in M_{n}}\alpha\chi_0^{-1}(f(0))\beta\chi_0^{-1}(f(1))\gamma\chi_0^{-1}(R(G,f))\dd f\\
    	=&\int_F\psi (a ) \alpha\beta(a)\gamma(R(G,a))|a|^{-1}_F\dd a\int_{f\in M_{n}}\alpha\chi_0^{-1}(f(0))\beta\chi_0^{-1}(f(1))\gamma\chi_0^{-1}(R(G,f))\dd f\\
    	=&\int_{M_n\times F}\psi(a)\alpha\chi_0^{-1}(af(0))\beta\chi_0^{-1}(af(1))\gamma\chi_0^{-1}(R(G,af))|a|_F^{n}\dd f\dd a\\
    	=&\int_{F_S}\psi(a)\alpha\gamma^{-1}(h(0))\beta\gamma^{-1}(h(1))\gamma\chi_0^{-1}(R(S,h))\dd_S h\ (\text{where $a$ is the highest coefficience of $h$})\\
    	=&\chi_0(\Delta(S))^{-\frac1{2}}\int_{F_S}\psi(a)\alpha\gamma^{-1}(h(0))\beta\gamma^{-1}(h(1))\gamma\chi_0^{-1}(R(S,h))\dd_* h \ (\text{Lemma }\ref{lemma2.3})\\
    	=&\chi_0(\Delta(S))^{-\frac1{2}}\int_{F_S}\psi\left(\sum_{i=1}^k\sum_{j=1}^{d_i}\frac{h(\alpha_{ij})}{\alpha_{ij}(\alpha_{ij}-1)G'(\alpha_{ij})}+\frac{h(0)}{-G(0)}+\frac{h(1)}{G(1)}\right)\\
    	&\times\alpha\gamma^{-1}(h(0))\beta\gamma^{-1}(h(1))\gamma\chi_0^{-1}\left(h(0)h(1)\prod_{i=1}^k\prod_{j=1}^{d_i}h(\alpha_{ij})\right)\dd_* h.
    	\end{split}
    	\end{equation}
    	Here we are using the Lagrange interpolation in the last equality to get the expression of $a$ in terms of $h$. Now via the isomorphism:
    	\[
    	F_S\cong F\times F\times F_G,\ h\mapsto (h(0),h(1),h\text{ mod }{G}), \dd_* h\mapsto\dd_* h(0)\cdot\dd_* h(1)\cdot \dd_{*G} h.
    	\]
    	where $d_{*G}h$ is the haar measure on $F_G$ such that $\mathcal{F}_G$ is an isometry, we have
    	\begin{equation}
    	\begin{split}
    	&\int_{F_S}\psi(a)\alpha\gamma^{-1}(h(0))\beta\gamma^{-1}(h(1))\gamma\chi_0^{-1}(R(S,h))\dd_* h\\
    	=&\int_{F_G}\psi\left(\sum_{i=1}^k\sum_{j=1}^{d_i}\frac{h(\alpha_{ij})}{\alpha_{ij}(\alpha_{ij}-1)G'(\alpha_{ij})}\right)\gamma\chi_0^{-1}\left(\prod_{i=1}^k\prod_{j=1}^{d_i}h(\alpha_{ij})\right)\dd_{*G} h\\
    	&\times \int_F \psi\left( \frac{h(0)}{-G(0)}\right) \alpha\chi_0^{-1}(h(0))\dd_* h(0)\int_F \psi\left( \frac{h(1)}{G(1)}\right) \beta\chi_0^{-1}(h(1))\dd_* h(1)\\
    	=&\int_{\prod_{i=1}^kF_{g_i}}\prod_{i=1}^k\psi\left(\Tr_{g_i}\frac{\varphi_i(h)}{\varphi_i(x(x-1)G')}\right)\prod_{i=1}^k\gamma\chi_0^{-1}\left(\N_{g_i}(\varphi_i(h)\right)\textstyle\prod\limits_{i=1}^k\dd_{i} \varphi_i(h)\\
    	&\times \int_F \psi\left( \frac{h(0)}{-G(0)}\right) \alpha\chi_0^{-1}(h(0))\dd_* h(0)\int_F \psi\left( \frac{h(1)}{G(1)}\right) \beta\chi_0^{-1}(h(1))\dd_* h(1)
    	\end{split}
    	\end{equation}
    	where $\varphi_i(x(x-1)G')$ should be interpreted as $\varphi_i\left(x(x-1)G'\text{ mod }G\right)$. Now make several change of variables:
    	\begin{equation}
    	\begin{split}
    	&\varphi_i(h)\mapsto \varphi_i(h)/\varphi_i(x(x-1)G'), 1\le i\le k;\\
    	&h(0)\mapsto -h(0)/G(0); \\
    	&h(1)\mapsto h(1)/G(1).
    	\end{split}
    	\end{equation}
    	We get:
    	\begin{equation}
    	\begin{split}    	
    	&\Gamma(\alpha\beta\gamma^{n-1})\int_{f\in M_{n}}\alpha\chi_0^{-1}(f(0))\beta\chi_0^{-1}(f(1))\gamma\chi_0^{-1}(R(G,f))\dd f\\
    	=&\chi_0(\Delta(S))^{-\frac{1}{2}}\prod_{i=1}^k\gamma\chi_0^{-1}(\N_{g_i}(\varphi_i(x(x-1)G')))|\varphi_i(x(x-1)G')|_{F_{g_i}}\prod_{i=1}^{k}\Gamma_{g_i}(\gamma)\\
    	&\times\alpha\chi_0^{-1}(-G(0))|G(0)|_F\beta\chi_0^{-1}(G(1))|G(1)|_F\Gamma(\alpha)\Gamma(\beta)\\ 	
    	 =&\chi_0(\Delta(S))^{-\frac{1}{2}}\gamma\chi_0^{-1}\left(\prod_{i=1}^k\prod_{j=1}^{d_i}\alpha_{ij}(\alpha_{ij}-1)G'(\alpha_{ij})\right)\left|\prod_{i=1}^k\prod_{j=1}^{d_i}\alpha_{ij}(\alpha_{ij}-1)G'(\alpha_{ij})\right|_F\prod_{i=1}^{k}\Gamma_{g_i}(\gamma)\\
    	&\times\alpha(-G(0))\beta(G(1))\Gamma(\alpha)\Gamma(\beta)\\
    	=&\chi_0\left(G(0)^2G(1)^2\Delta(G)\right)^{-\frac{1}{2}}\gamma\chi_0^{-1}\left(G(0)G(1)R(G,G')\right)\left|G(0)G(1)R(G,G')\right|_F\prod_{i=1}^{k}\Gamma_{g_i}(\gamma)\\
    	&\times\alpha(-G(0))\beta(G(1))\Gamma(\alpha)\Gamma(\beta)\\
    	=&\alpha(-1)\chi_0(\Delta(G))^{-\frac1{2}}\alpha\gamma\chi_0^{-1}(G(0))\beta\gamma\chi_0^{-1}(G(1))\gamma(R(G,G'))\Gamma(\alpha)\Gamma(\beta)\prod_{i=1}^{k}\Gamma_{g_i}(\gamma).
    	\end{split}
    	\end{equation}
        which completes the proof of the proposition.
    \end{proof}

\section{Proof of Theorem 1.1}\label{section3}

    \begin{proof}
    	The proof is by induction. When $n=1$, the formula (\ref{formula}) is just (\ref{2.4}).
    	Now assume $S_{n-1}(\alpha,\beta,\gamma)$ converges for any $\alpha,\beta,\gamma$ in the region $R_{n-1}$ and we have
        \begin{equation}
        S_{n-1}(\alpha,\beta,\gamma)=\prod_{j=0}^{n-2}\frac{\Gamma(\alpha\gamma^{j})\Gamma(\beta\gamma^{j})\Gamma(\gamma^{j+1})}{\Gamma(\alpha\beta\gamma^{n+j-1})\Gamma(\gamma)}.
        \end{equation}
        Consider the double integral:
        \begin{equation}
        T_n:=\int_{P\in M_{n-1}}\int_{Q\in M_{n}}\alpha\chi_0^{-1}(Q(0))\beta\chi_0^{-1}(Q(1))\gamma\chi_0^{-1}(R(P,Q))\dd Q\dd P.
        \end{equation}
        By Propostion \ref{prop2}, we get
        \begin{equation}\label{T1}
        \begin{split}
        &T_{n,P}\\
        :=&\int_{P\in M_{n-1}}\left(\int_{Q\in M_{n}}\alpha\chi_0^{-1}(Q(0))\beta\chi_0^{-1}(Q(1))\gamma\chi_0^{-1}(R(P,Q))\dd Q\right)\dd P\\
        =&\int_{ P\in M_{n-1}}\alpha(-1)\chi_0^{-\frac1{2}}(\Delta(P))\alpha\gamma\chi_0^{-1}(P(0))\beta\gamma\chi_0^{-1}(P(1))\gamma(R(P,P'))\frac{\Gamma(\alpha)\Gamma(\beta)\prod\limits_{i=1}^{k}\Gamma_{g_i}(\gamma)}{\Gamma(\alpha\beta\gamma^{n-1})}\dd P\\
        =&\alpha^n(-1)S_{n-1}(\alpha\gamma,\beta\gamma,\gamma)\frac{\Gamma(\alpha)\Gamma(\beta)\Gamma(\gamma)^{n-1}}{\Gamma(\alpha\beta\gamma^{n-1})}
        \end{split}
        \end{equation}
        valid in the region
    	\begin{equation}
    	\begin{split}
    	&\re \alpha,\re\beta,\re\gamma>0;\\
    	&\re \alpha+\re\beta+(n-1)\re\gamma<1;\\
    	&\re \alpha+1+\re\beta+1+2(n-2)\re\gamma<1.
    	\end{split}
    	\end{equation}
    	which is exactly the region $R_n$. Note that the absolute value of the intergrand is the same as replacing $\alpha,\beta,\gamma$ with $\chi_0^{\re \alpha},\chi_0^{\re \beta},\chi_0^{\re \gamma}$. Hence the we have $T_n=T_{n,P}=T_{n,Q}$ in $R_n$ by Fubini-Tonelli theorem, where $T_{n,Q}$ is defined by the following
    	\begin{equation}\label{T2}
    	\begin{split}
    	&T_{n,Q}\\
    	:=&\int_{Q\in M_{n}}\left(\int_{P\in M_{n-1}}\alpha\chi_0^{-1}(Q(0))\beta\chi_0^{-1}(Q(1))\gamma\chi_0^{-1}(R(P,Q))\dd P\right)\dd Q.\\
    	=&\int_{Q\in M_n}\alpha\chi_0^{-1}(Q(0))\beta\chi_0^{-1}(Q(1))\chi_0(\Delta(Q))^{-\frac1{2}}\gamma(R(Q,Q'))\frac{\prod_{i=1}^{k}\Gamma_{g_i}(\gamma)}{\Gamma(\gamma^n)}\dd Q\\
    	=&\alpha^n(-1)S_n(\alpha,\beta,\gamma)\frac{\Gamma(\gamma)^n}{\Gamma(\gamma^{n})}
    	\end{split}
    	\end{equation}
    	here we are using Proposition \ref{prop1} in the second equation. Thus $S_n(\alpha,\beta,\gamma)$ converges in the region $R_n$, and $T_{n,P}=T_{n,Q}$ gives us
    	\begin{equation}\label{ind equ}
    	S_n(\alpha,\beta,\gamma)
    	 =S_{n-1}(\alpha\gamma,\beta\gamma,\gamma)\frac{\Gamma(\alpha)\Gamma(\beta)\Gamma(\gamma^n)}{\Gamma(\alpha\beta\gamma^{n-1})\Gamma(\gamma)}
    	 =\prod_{j=0}^{n-1}\frac{\Gamma(\alpha\gamma^{j})\Gamma(\beta\gamma^{j})\Gamma(\gamma^{j+1})}{\Gamma(\alpha\beta\gamma^{n+j-1})\Gamma(\gamma)}.
    	\end{equation}
    	which completes the proof of the main theorem.
    \end{proof}

    As a special case, we consider when $F=\C$. We take $\psi ( z ) =  \psi ( x + i y ) = e^{ 4\pi i x}$ as in \cite{T}.
     The self-dual Haar measure $ d z $ is then twice the usual Lebesgue measure on ${\mathbb{C}}$. Note that $ | z |_{\mathbb{C}} = | z |^2 $.
    The map
    \begin{equation}
    \Phi_n: \C^n\rightarrow M_n,\ (z_1,\cdots,z_n)\mapsto f(z)=\prod_{i=1}^n(z-z_i) = z^n + b_{n-1} z^{n-1} + \dots + b_0 
    \end{equation}
    is surjective and generically $n!$ to $1$. 
     It is known that the Jacobian of this map is  $\text{Jac}\,\Phi_n(z_1,\cdots,z_n) = \prod_{i<j}(z_i-z_j)$ \cite{Ser}.
        Thus
    \begin{equation}
    |\text{Jac}\,\Phi_n(z_1,\cdots,z_n)|_{\mathbb{C}}=|\prod_{i<j}(z_i-z_j)|_{\mathbb{C}}=|\Delta(f)|_{\mathbb{C}}^{\frac12}.
    \end{equation} So we have
    \begin{equation}
    \begin{split}
    S_n(\alpha,\beta,\gamma)
    =\ &\alpha^n\gamma^{\frac{n(n-1)}{2}}(-1) \int_{f\in M_{n}}\alpha\chi_0^{-1}(f(0))\beta\chi_0^{-1}(f(1))\gamma\chi_0^{-\frac1{2}}(\Delta(f))\dd f\\
    =\ &\frac 1 { n!} \int_{\C^n}\alpha\chi_0^{-1}(\prod_{i=1}^nz_i)\beta\chi_0^{-1}(\prod_{i=1}^n(1-z_i))\gamma(\prod_{i\neq j}(z_i-z_j))\textstyle\prod\limits_{i=0}^n\dd z_i.
    \end{split}
    \end{equation}
    Then by Thm \ref{thm}, we get:
    \begin{equation}\label{gencomplex} \int_{\C^n}\prod_{i=1}^n\alpha\chi_0^{-1}(z_i)\beta\chi_0^{-1}(1-z_i)\prod_{ i \neq j}\gamma((z_i-z_j))\textstyle\prod\limits_{i=0}^n\dd z_i\displaystyle
    =  n!\prod_{j=0}^{n-1}\frac{\Gamma_{\C}(\alpha\gamma^{j})\Gamma_{\C}(\beta\gamma^{j})\Gamma_{\C}(\gamma^{j+1})}{\Gamma_{\C}(\alpha\beta\gamma^{n+j-1})\Gamma_{\C}(\gamma)}.
    \end{equation}
    valid in the region $R_n$.
    If $\alpha,\beta,\gamma$ are all unramified, i.e. $\alpha=|\cdot|_{\mathbb{C}}^{a},\beta = |\cdot|_{\mathbb{C}}^{b},\gamma = |\cdot|_{\mathbb{C}}^{c}$, for some $a,b,c\in \C$, then formula (\ref{gencomplex}) becomes:
    \begin{equation}\label{complex}
    \begin{split}
    &\int_{\C^n}\prod_{i=1}^n|z_i|^{2a-2}|1-z_i|^{2b-2}\prod_{ 1\leq i < j\leq n }|z_i-z_j|^{4 c}\textstyle\prod\limits_{i=0}^n\dd z_i\displaystyle\\
    =\,& n!\prod_{j=0}^{n-1}\frac{\Gamma_{\C}(a+jc)\Gamma_{\C}(b+jc)\Gamma_{\C}((j+1)c)}{\Gamma_{\C}(a+b+(n+j-1)c)\Gamma_{\C}(c)}\\
    =\,&\frac{\prod\limits_{j=0}^{n-1}2\sin(\pi(a+jc))\sin(\pi(b+jc))\sin(\pi(j+1)c)}{n
    	!\prod\limits_{j=0}^{n-1}\sin(\pi(a+b+(n+j-1)c))\sin(\pi c)}S_n(a,b,c)^2
    \end{split}
    \end{equation}
    where for any $s\in \C$ with $0<\re s<1$,
    \begin{equation}
    \Gamma_\C(s):=\Gamma_\C(|\cdot|_{\mathbb{C}}^s)
    = \frac{(2\pi)^{1-s}\Gamma(s)}{(2\pi)^{s}\Gamma(1-s)}
    = 2^{1-2s}\pi^{-2s}\Gamma(s)^2\sin(\pi s).
    \end{equation}
    (cf. \cite{T}) and $S_n(a,b,c)$ is defined in (\ref{selberg}). The formula (\ref{complex}) is the same as the formula obtained by Aomoto \cite{Ao}.

\

\end{document}